\documentclass[11pt]{article}

\usepackage{fullpage}
\usepackage{amsmath,amssymb, amsthm}
\usepackage{ulem}
\usepackage{cite}
\usepackage{authblk}

\DeclareMathOperator{\sol}{div}						 
	
\DeclareMathOperator{\curl}{curl}
\newcommand{\norm}[1]{\left\|#1\right\|}   
		  
\newcommand{\derivuj}[2]{\frac{\partial{#1}}{\partial{{#2}}}}

\newcommand{\abs}[1]{\left\lvert#1\right\rvert}      
\newcommand{\field}[1]{\mathbb{#1}}
\def\det{\partial_t}
\def\R{\field{R}}      
\def\N{\field{N}} 
 
\def\T{\field{T}} 
\def\mfi{\varphi}
\def\S{\field{S}} 
\def\B{\mathcal{B}}
\def\Te{\mathcal{T}}
\newcommand{\rntou}[1]{\R^{#1}}    

\def\carka{\sp{\prime}}
\def\carky{\sp{\prime\prime}}
\def\rtri{\rntou{3}}

\def\Id{\field{I}} 
\def\epsil{\varepsilon}
\newcommand{\vektor}[1]{{\mathbf{#1}}}
\def\ro{\varrho}

\def\en{\vektor{n}}
\def\bfi{\boldsymbol\mfi}
\def\Phib{\boldsymbol\Phi}

\def\ef{\vektor{F}}
\def\A{\vektor{A}}
\def\Ha{\vektor{H}}
\def\u{\vektor{v}}
\newcommand{\de}[1]{\mathrm{d}{#1}}     
	
\def\dx{\de{x}}

\newcommand{\inte}[1]{\int\limits_{\Omega}{#1}\dx  }
\newcommand{\intek}[1]{\int\limits_{\Omega_K}{#1}\dx  }

\newcommand{\refx}[1]{(\ref{#1})}
\def\qqquad{\quad\quad\quad}
\def\weak{\rightharpoonup}
\def\weaks{\rightharpoonup ^{*}}

\def\sob{W^{1,2}(\Omega)}

\def\ksi{\zeta}

\theoremstyle{plain}
\newtheorem{thm}{Theorem}
\newtheorem{lem}{Lemma}
\newtheorem{prop}{Proposition}
\newtheorem*{LSFPT}{Schauder fixed point theorem}
\theoremstyle{definition}
\newtheorem{defin}{Definition}
\newtheorem{note}{Remark}

\newcommand{\lnorm}[2]{\norm{#1}_{L^{#2}(\Omega)}}
\newcommand{\lknorm}[2]{\norm{#1}_{L^{#2}(\Omega_K)}}
\newcommand{\lnorma}[2]{\norm{#1}_{L^{#2}}}

\title{Decently regular steady solutions 
 to the compressible NSAC system}
\author[1,2]{\v{S}imon Axmann\thanks{axmas7am@karlin.mff.cuni.cz}}

\author[3]{Piotr Bogus\l{}aw Mucha\thanks{p.mucha@mimuw.edu.pl; Corresponding author}\small}

\affil[1]{Mathematical Institute of Charles University, Prague}

\affil[2]{Department of Mathematics, University of Chemistry and Technology, Prague}

\affil[3]{Institute of Applied Mathematics and Mechanics, University of Warsaw}

\begin{document}
\maketitle

{\bf Abstract.} We aim at proving existence of weak solutions to the stationary compressible Navier-Stokes system coupled with
the Allen-Cahn equation. The model is studied in a bounded three dimensional domain with slip boundary conditions for the momentum equations and
the Neumann condition for the Allen-Cahn model. The main result establishes existence of weak solutions with bounded densities. The construction 
is possible assuming sufficiently large value of the heat capacity ratio $\gamma$ ($p \sim \ro^\gamma$). As a corollary we obtain weak solutions for a less restrictive case losing 
point-wise boundedness of the density. 

\section{Introduction}

Phase transition phenomena are important subject of applied mathematics. Complexity of physical nature of these processes makes us to choose different models in 
order to obtain the best description in a concrete studied case. Here we want to concentrate our attention on phenomena with fuzzy phase interfaces. We  think
about the process of melting or freezing, at the level of almost constant critical temperature. We want to control densities of phase constituents, and here we arrive at 
the Allen-Cahn equation:
\begin{equation}
 \ro \derivuj{f}{c}(\ro,c) - \Delta c =\ro \mu.
\end{equation}
The equation above gives us control of one of the phase constituents in terms of the chemical potential $\mu$. In the chosen model this equation is coupled with the compressible 
Navier-Stokes system, which represents a well established and frequently studied model describing the flows of viscous compressible single constituted fluids.

The mathematical analysis of coupled systems: the compressible Navier-Stokes type and the phase separation
 is in its infancy \cite{AbFe08,Feetal10,Kots12,KoZa09,DiLiLu13, DiLi13}, although the mathematical theory of each of them separately is quite developed 
(see e.g. \cite{Feir04}, \cite{GrPeSc06,Sc00}). In this article, we study existence  of steady weak solutions.
 The thermodynamically consistent derivation of the model under consideration, which is a variant of a model proposed by Blesgen \cite{Bles99}, was presented by Heida, M\'{a}lek and Rajagopal 
in \cite{HeMaRa12}. It is represented by the following system of 
partial differential equations for three unknowns: density  of the fluid $\ro$, velocity field $\u$, and concentration of one selected constituent $c$
\begin{gather}
     \sol(\ro \u) = 0, \label{RK} \\ 
	 \sol(\ro \u\otimes\u) =\sol \T+\ro\ef  , \\
     \sol(\ro c \u) =  - \mu, \label{third} \\
     \ro \mu = \ro \derivuj{f}{c}(\ro,c) - \Delta c , \label{system}
\end{gather} \def\stat{\refx{RK}-\refx{system}}
 where the stress tensor $\T$ is given by $$\T(\nabla\u,\nabla c,\ro,c) = \S(\nabla\u) -\bigl(\nabla c \otimes \nabla c -\frac{\abs{\nabla c}^2}{2}\Id\bigr)- p(\ro,c) \Id ,$$
 with the thermodynamical pressure $p(\ro,c) =\ro^2 \derivuj{f}{\ro} $ and the viscous stress $\S$ satisfying the Stokes law for Newtonian fluid 
 $$\S(\nabla \u) = \nu\left(\nabla\u +\nabla^T\u - \sol\u\Id \right)+ \eta\sol\u\Id, $$ 
  with the shear viscosity coefficient $\nu>0$, and the bulk viscosity coefficient $\eta\geq 0 $.
The free energy is assumed in the form with the so-called logarithmic potential\footnote{We could assume also more general pressure law function similarly to \cite{singular}, but it would lead only to
 additional unnecessary technicalities, so we omit it. Similarly, we can consider more general singular functions $L(c).$} 
$$	f(\ro,c) = \frac{1}{\gamma-1} \ro^{\gamma-1} +\bigl( a_1 c +a_2 (1-c)\bigr)\log \ro +c\log c + (1-c) \log(1-c) + b(c) $$
 with some $\gamma>1$\footnote{The values of $\gamma$ which ensure the existence result will be specified later.}, $a_1,a_2\geq0$ and $b$ a smooth bounded function 
with $\abs{b\carka(c)}\leq C$. Moreover we assume without loss of generality that $a_{1}\geq a_{2},$ and we denote for the sake of simplicity $a=a_1-a_2$, $d=a_1$, and $L(c) =c\log c + (1-c) \log(1-c).$  The  logarithmic terms related to 
the entropy of the system assure that $c\in[0,1]$ almost everywhere, since
\begin{align*}	\derivuj{f}{c}(\ro,c)=&  \log c - \log(1-c)  + (a_1-a_2)\log\ro +b\carka(c)\\
           =&L\carka (c)+ a\log\ro+b\carka(c),\\
	p(\ro,c) =& \ro^\gamma + \ro \bigl(ac+d\bigr).
\end{align*}
 
The fluid is contained in a smooth bounded domain $\Omega$, we supply the equations in the domain with boundary conditions
\begin{gather}
\u\cdot\en =0\text{ at }\partial\Omega,\label{imper}\\
\en\cdot \T(c,\nabla\u)\cdot\boldsymbol{\tau}_n+k\u\cdot\boldsymbol{\tau}_n =0\text{ at }\partial\Omega,\\
\nabla c\cdot\en =0\text{ at }\partial\Omega \label{tokh},
\end{gather}
where parameter $k> 0$ represents the friction on the boundary\footnote{For slip boundary condition corresponding to the case $k=0$, we need to assume that the domain $\Omega$ is not axially symmetric.},
 $\boldsymbol{\tau }_n,\:n=1,2$ are two linearly independent tangent vectors to $\partial\Omega$, and $\en$ denotes the normal vector.
 
 The solutions are parametrized by means of the condition
\begin{equation}
\inte{\ro} = M.\label{conserv}
\end{equation}   
The fluid is driven by an external force $\ef\in L^{\infty}(\Omega,\rtri)$.

The main result of the paper is the following.

\smallskip

\begin{thm}\label{main}
 Let $\gamma > 6$, $M>0$ and $\ef\in L^{\infty}(\Omega,\rtri)$. Then there exists at least one weak solution to the system {\stat} such that $c\in[0,1]$ in $\Omega$
\begin{equation}
 \ro \in L^\infty(\Omega), \quad \u \in W^{1,p}(\Omega,\rtri)  \mbox{ and } c \in W^{2,p}(\Omega) \mbox{ for } p< \infty.
\end{equation}
In addition, if we assume only $\gamma > 3$, then there exists at least one weak solution such that
\begin{equation}
  \ro\in L^{3\gamma-6}(\Omega), \quad\u\in W^{1,2}(\Omega), \quad \nabla c \in L^{\frac{6\gamma-12}{\gamma}}(\Omega).
\end{equation}
The meaning of the solutions is explained in Definition \ref{def}, below.
\end{thm}

\smallskip

The existence of non-stationary solutions to the above mentioned system with no-slip boundary condition for the velocity was shown by Feireisl et al.
 in \cite{Feetal10}. Ding et al.  studied in  \cite{DiLiLu13} the global existence of weak, strong, and even classical solutions in 1D with free energy approximated
 by a suitable bistable function, assuming no vacuum zones in the initial data.
Further, Kotschote \cite{Kots12} put his attention to a more advanced model where the extra stress tensor is multiplied by density function. He showed, however only, 
the local-in-time existence of strong solutions
provided positiveness of density, including the thermal effects as well. The existence of travelling waves for this model was shown by Freist\"{u}hler in \cite{Frei14}. 
To best our knowledge there are no results concerning existence of weak solutions to the stationary system \refx{RK}-\refx{system}.
  Concerning the  steady solutions to the compressible Navier-Stokes system, we refer to the pioneering work of Lions \cite{Lion98}, and further
 extensions \cite{NoSt04, PlSo07, BrNo08, FrStWe10, JiZh11, PiasJDE, Zat1}, we use a technique developed by Mucha, and Pokorn\'{y} (see \cite{MuPo06, MuPo08}),
 which was modified for the Navier-Stokes-Fourier system \cite{MuPo09}, too. Methods from \cite{MuPo06} allow to obtain solutions with point-wise bounded densities, which
seems to be the best possible regularity for weak solutions. Here we find an explanation of titled {\it decently regular} solutions.

We introduce the following definition of weak solutions to the system \stat.

\smallskip 

\begin{defin}
\label{def}
   Let $M>0$ be a given constant, $\gamma>3$, we say that quadruple $\ro,\u,\mu, c$ is a  weak solution 
 to the steady Navier-Stokes-Allen-Cahn system, if $\ro\in L^{\gamma}(\Omega),
\:\ro\geq0\text{ a.e.  in }\Omega $, $\u\in   W^{1,2}(\Omega)$, $\mu\in L^{2}(\Omega)$, $c\in W^{1,2}(\Omega),$ $0\leq  c \leq 1 \text{ a.e. on }\{\ro>0\} $, 
with $\ro L\carka(c)\in L^{2\gamma/(\gamma+1)}(\Omega)$ and $\u\cdot\en = 0$ satisfied on $\partial\Omega$ in the sense of traces, and if the following holds true:
   \begin{enumerate}
   \item Continuity equation is satisfied in the distributional sense, id est:\footnote{Note that we do not need the notion of renormalized continuity equation, 
since our density is regular enough.}
   \begin{equation*}
   \sol\bigl(\ro\u\bigr)=0,\text{ in }\mathcal{D}\carka(\rtri)
   \end{equation*}
   \item For every $\bfi\in C^\infty(\overline{\Omega},\rtri),\:\bfi\cdot\en = 0$ at $\partial \Omega$
   \begin{equation*}\inte{ \bigl(-\ro(\u\otimes\u):\nabla\bfi +  \T(\ro,c,\nabla c,\nabla\u) :\nabla\bfi\:\bigr)}  
   + \int\limits_{\partial\Omega} k(\u \cdot \boldsymbol{\tau} )(\bfi\cdot\boldsymbol{\tau}) \de{S}  =\inte{\ro\ef\cdot\bfi\:}.
   \end{equation*}
   \item For every $\mfi\in C^\infty(\overline{\Omega},\R)$
   \begin{equation*}
   \inte{\ro \u\cdot{\nabla c }\, \mfi\,}=\inte{-\mu \mfi\,}
   \end{equation*}
   and
   \begin{equation}
   \inte{\ro\mu \mfi\,}=\inte{\ro\derivuj{f(\ro,c)}{c}\mfi + \nabla c\cdot \nabla \mfi\,}.\label{mu}
   \end{equation}
   \end{enumerate}
   such that $\inte{\ro\:}=M$. 
\end{defin}

\smallskip 

The paper is constructed as follows. First we compute formal a priori estimate, it allows us to determine expected regularity of sought solutions. In Section \ref{aproxsec}
we deal with the approximative system, we construct regular approximative solutions together with required estimates in the dependence of  approximative
parameters. Finally we analyze the limit, showing the strong convergence  of approximative densities. The method is based on the fact that for the approximative
densities $\ro_\epsilon$ we are able to find such $m$ that
\begin{equation}
 \lim_{\epsilon \to 0} |\{ \ro_\epsilon > m\}| =0.
\end{equation}
We skip the proof of the second part of Theorem \ref{main}, we explain it in Remark \ref{smaller}. Through out the paper we try to follow the standard notation.

\section{A priori estimates}
Before the technical part of the proof, we will present here the a priori estimates on certain norms of the solution derived by purely heuristic approach. 
All generic constants, which may depend on the given data as well, will be denoted by $C$, its values can vary from line to line or even in the same formula. 

\begin{lem}\label{basic}
Let $\Omega\subset\R^{3}$ be a bounded domain with $C^{2+\zeta}$ boundary. Assume that all the above mentioned hypotheses are satisfied with $\gamma>3$ and  that $\ro,\u,\mu,c$ is a sufficiently 
smooth solution satisfying \refx{RK}-\refx{system} with the boundary conditions \refx{imper}-\refx{tokh}. Then
\begin{align}
 \norm{\u}_{W^{1,2}(\Omega)}+ \lnorm{\ro}{3\gamma-6}+\lnorm{\nabla c}{\frac{6\gamma-12}{\gamma}}+\lnorm{\mu}{2}+\lnorm{\ro L\carka(c)}{\frac{6\gamma-12}{3\gamma-4}} \leq C,
\end{align}
where $C$ may depend on the data, but is independent of the solution.
\end{lem}
\begin{note}
Note that ${\frac{6\gamma-12}{\gamma}}>2,$ for $\gamma>3$ and that from the bound of the last norm on the left hand side, we  immediately deduce that $c\in[0,1]$ a.e. on 
the set $\{\ro>0\}.$ Moreover, if $\gamma>4$, then according to the Sobolev imbedding $c$ is continuous function and we conclude 
from the maximum principle for harmonic functions that in fact $c\in[0,1]$ a.e. in $\Omega.$
\end{note}
\begin{proof}
First, multiplying the momentum equation by the velocity field $\u$ yields (with usage of Korn's inequality and boundary condition)
\begin{equation}-\inte{p\sol\u}+C \inte{\abs{\nabla\u}^2}\leq \inte{\ro\mu\nabla c\cdot \u- \ro \derivuj{f}{c}\nabla c \cdot \u+\ro\ef\cdot\u\:}\label{gradu},
\end{equation}
where we have used the equation \refx{system} as well.
Further, we conclude from the continuity equation that 
\begin{equation}
\inte{p\sol\u + \sol(\ro f\u)} = \inte{\ro \derivuj{f}{c}\nabla c\cdot \u },
\end{equation} 
and according to the constitutive equation for $\mu$ we get
\begin{equation}
\inte{{\mu^2}} = \inte{-\ro\mu \nabla c \cdot \u}\label{gradmu}.
\end{equation}
Thus, summing up \refx{gradu}-\refx{gradmu} yields for $p>\frac{6}{5}$
$$
\inte{\abs{\nabla\u}^2+\mu^2}\leq C\inte{ \ro\ef\cdot\u }\leq C\Bigl(1+ \lnorm{\ro}{\frac{6}{5}}^2\Bigr)\leq C\biggl(1+ \lnorm{\ro}{p}^{\frac{p}{3(p-1)}}\biggr).
$$

In order to bound the density by means of the Bogovskii estimates, we need  $\nabla c$ in $L^{q}(\Omega),\: q>2.$ This can be deduced from the constitutive relation for $\mu$. We state this for purpose of 
future references more generally in the following lemma.
\begin{lem}  \label{koncentr}
 Suppose that $f(\ro,c)$ is as above, and $\mu\in L^{q}(\Omega)$ and $\ro\in L^p(\Omega)$ with $q\geq 2$, $p> 3 $, $q<\frac{3p}{p-3}$ satisfy the equation \refx{system} with the boundary condition \refx{tokh}. Then
 $$\lnorm{\ro\derivuj{f}{c}}{\frac{pq}{p+q}}+\lnorm{\ro\mu}{\frac{pq}{p+q}}+\lnorm{\Delta{c}}{\frac{pq}{p+q}} + \lnorm{\nabla c}{\frac{3pq}{3p+3q-pq}} \leq C\bigl( \lnorm{\ro}{p}(\lnorm{\mu}{q}+1)+1\bigr).$$
\end{lem}
\begin{proof}[Proof of Lemma \ref{koncentr}]
Let us test the corresponding equation by $F\bigl(L\carka(c)\bigr)$, for increasing function $F$ with growth $\ksi F(\ksi)\sim \abs{\ksi} ^{\beta+1}$ with some $\beta>0$,  so (recall that $L$ is convex)
\begin{equation*}
\begin{split}
& \lnorm{\ro^{1/(\beta+1)}L\carka(c)}{\beta+1} ^{\beta+1}\leq\inte{\abs{\nabla c}^2 F\carka\bigl(L\carka(c)\bigr) L\carky(c) + \ro L\carka(c)F\bigl(L\carka(c)\bigr) }=\\ &\qquad\quad=
\inte{\ro\bigl(\mu  - a\log\ro - b\carka(c)\bigr) F\bigl(L\carka(c)\bigr)}\\ &\qquad\quad\leq C
\inte{\ro^{\frac{\beta}{\beta+1}} \abs{L\carka(c)}^{\beta}\bigl(\abs{\mu}  \ro^{\frac{1}{\beta+1}} + \ro^{\frac{1}{\beta+1}}\abs{\log\ro} \bigr) + \ro \abs{L\carka(c)}^{\beta} \abs{c}}\\
&\quad\qquad\leq C\biggl( \lnorma{\ro^{\frac{1}{\beta+1}}L\carka(c)}{\beta+1} ^{\beta} \Bigl(\lnorma{\mu}{q}\lnorma{\ro}{\frac{q}{q-(\beta+1)}} ^{\frac{1}{\beta+1}} + 
\Bigl(\inte{\ro \abs{\log\ro}^{\beta+1}}\Bigr)^{\frac{1}{\beta+1}}\Bigr) +   \lnorma{\ro^{\frac{1}{\beta+1}}L\carka(c)}{\beta+1} ^{\beta+\frac{1}{2}} \biggr)
\end{split}
\end{equation*}
which yields
\begin{equation*}
\lnorm{\ro^{1/(\beta+1)} L\carka(c) }{\beta+1}\leq \:
 C\biggl( 1+ \lnorma{\mu}{q}\lnorma{\ro}{\frac{q}{q-(\beta+1)}} ^{\frac{1}{\beta+1}} + \Bigl(\inte{\ro \abs{\log\ro}^{\beta+1}}\Bigr)^{\frac{1}{\beta+1}} \biggr).
\end{equation*}
Now we fix the exponent $\beta$ such that $  \frac{q}{q-(\beta+1)}=p \in\bigl(3,\infty\bigr)$, hence 
$$\lnorm{\ro^{p/(pq-q)} L\carka(c) }{\frac{q(p-1)}{p}}\leq C\left(1+\lnorm{\mu}{q}\lnorm{\ro}{p}^{\frac{p}{q(p-1)}}   
+  \Bigl(\inte{\ro \abs{\log\ro}^{\frac{(p-1)q}{p}}}\Bigr)^{\frac{p}{q(p-1)}}  \right).$$
Thus,
\begin{equation*}
\begin{split}
\lnorm{\ro L\carka(c)}{\frac{pq}{p+q}}=& \lnorm{\ro^{p/(pq-q)}\ro^{(pq-q-p)/(pq-q)}L\carka(c)}{\frac{pq}{q+p}}\\
\leq &\lnorm{\ro^{p/(pq-q)}L \carka(c)}{\frac{q(p-1)}{p}}
\lnorm{\ro^{(pq-q-p)/(pq-q)}}{\frac{pq(q-1)}{pq-p-q}}\\
\leq & C \lnorm{\ro}{p}^{\frac{pq-p-q}{q(p-1)}}\biggl(1+\lnorm{\mu}{q}\lnorm{\ro}{p}^{\frac{p}{q(p-1)}}   +  \Bigl(\inte{\ro \abs{\log\ro}^{\frac{(p-1)q}{p}}}\Bigr)^{\frac{p}{q(p-1)}}  \biggr)\\
\leq & C \lnorm{\ro}{p}^{\frac{pq-p-q}{q(p-1)}}\left(1+\lnorm{\mu}{q}\lnorm{\ro}{p}^{\frac{p}{q(p-1)}}   + \lnorm{\ro}{p}^{\frac{p}{q(p-1)}}\right)\\
\leq & C \left((\lnorm{\mu}{q}+1)\lnorm{\ro}{p}+1\right).
\end{split}
\end{equation*}
For the other terms we have
\begin{align*}\lnorm{\ro b\carka(c)}{\frac{pq}{p+q}}+
\lnorm{\ro\log\ro}{\frac{pq}{p+q}}&\leq C(1+\lnorm{\ro }{p}),\\
\lnorm{\ro\mu}{\frac{pq}{p+q}} &\leq \lnorm{\ro}{p}\lnorm{\mu}{q}.
\end{align*}

Using the classical elliptic estimates on equation $-\Delta c = \ro\mu - \ro \derivuj{f}{c},$ together with embedding $W^{1,\frac{pq}{p+q}}(\Omega)\hookrightarrow L^{\frac{3pq}{3p+3q-pq}}(\Omega)$ we get 
\begin{equation}
\lnorm{\nabla c}{\frac{3pq}{3p+3q-pq}}\leq \lnorm{\ro\derivuj{f}{c}+\ro\mu}{\frac{qp}{q+p}} + \norm{c}_{L^s(\tilde\Omega)}\leq C \left((\lnorm{\mu}{q}+1)\lnorm{\ro}{p}+1\right),
\end{equation}
where  for $\tilde\Omega$ we can take $\{\ro>\ro_0\}$ which has positive measure and on which $c\in[0,1]$ a.e. according to the logarithmic terms in $L\carka(c).$ This completes the proof of Lemma \ref{koncentr}.
\end{proof}
Applying Lemma \ref{koncentr} with $q=2$ yields
$ \lnorm{\nabla c}{\frac{6p}{6+p}}\leq C\bigl(1+\lnorm{\ro}{p}^{(7p-6)/(6p-6)}\bigr).$
Now, we are ready to perform the Bogovskii estimates, id est to test the momentum equation by\footnote{We use notation   $\left\{g\right\}_{\Omega}=\frac{1}{\abs{\Omega}}\int_{\Omega}{g}\:\dx $. }
$$\Phib = \B\left[  \ro^{\alpha}-\left\{ \ro^{\alpha}\right\}_{\Omega}  \right],$$
where $\alpha>0$ will be specified later, and $\B \sim \sol^{-1}$ is the Bogovskii operator. 
The theory implies that $\|\nabla \Phib\|_{L^p}\leq C\|\ro^{\alpha}\|_{L^p}$ with $\Phib|_{\partial \Omega}=0$ (see e.g. \cite{DaMu12, DaMu13, NoSt04}); we obtain
\begin{multline*}
	\inte {    p(\ro,c) \ro^{\alpha}     }
	\leq	\inte {  \bigl( -(\ro\u\otimes\u):\nabla\Phib+ \S(\nabla\u):\nabla\Phib - \ro \ef\cdot\Phib  \bigr) }\\
	+ C \inte{\abs{\nabla c}^2\abs{\nabla\Phib}\:} + \inte {   p(\ro,c)  \left\{\ro^{\alpha}\right\}_\Omega  }. 
	\end{multline*}
The terms on the left hand side of the inequality have sign and give the desired estimate of $\ro^{\gamma+\alpha}$, if the right hand side will be estimated, thus we set $p=\gamma+\alpha$. 
We will present only the most difficult and restrictive terms. 
\begin{multline}
	\inte {\abs{\nabla c }^2\abs{\nabla\Phib}} \leq 
	\lnorm{\nabla c}{\frac{6p}{6+p}}^2\lnorm{\nabla\Phib}{\frac{3p}{2p-6}}
\leq C \Bigl( 1+ \lnorm{\ro}{p}^{(7p-6)/(3p-3)}\lnorm{\ro^\alpha}{\frac{3p}{2p-6}} \Bigr)\\
\leq C \Bigl( 1+ \lnorm{\ro}{p}^{(7p-6)/(3p-3)}\lnorm{\ro}{p}^{\frac{3\alpha p-2p+6}{3(p-1)}} \Bigr)
\leq C \Bigl( 1+ \lnorm{\ro}{p}^{(5p+3\alpha p)/(3p-3)} \Bigr)\label{gradc},
\end{multline}
provided $\frac{3p\alpha}{2p-6}\leq p$, or equivalently \begin{equation}0<\alpha\leq 2(\gamma-3),\label{gammatri}\end{equation} which yields the restriction $\gamma>3.$
The condition
 $\frac{(\gamma+\alpha)(5+3\alpha)}{3(\gamma+\alpha)-3} < \gamma+\alpha$ is satisfied even for all $\gamma>\frac{8}{3},$ so we can put this term to the left hand side.
 
Further, provided $\alpha<2\gamma-3$
\begin{multline}
\inte { (\ro\u\otimes\u):\nabla\Phib} \leq \lnorm{\u}{6}^2 \lnorm{\ro\nabla\Phib}{{3}/{2}}\\
 \leq 
	C\bigl(1+\lnorm{\ro}{p}^{\frac{p}{3(p-1)}}\bigr) \lnorm{\ro}{\frac{3p}{2p-3\alpha}}\lnorm{\nabla\Phib}{\frac{p}{\alpha}}
 \leq 
	C\Bigl(1+\lnorm{\ro}{p}^{\frac{p}{3(p-1)}+\alpha}\Bigr)\lnorm{\ro}{p}^{\frac{p+3\alpha}{3(p-1)}},
	\end{multline}
  where the condition $	\frac{2(\gamma+\alpha)+3\alpha}{3(\gamma+\alpha-1)}+\alpha<\gamma+\alpha$
  is less restrictive since it requires only $5+3\alpha<3(\gamma+\alpha) \Rightarrow \gamma>\frac{5}{3}$. The other terms can be estimated similarly so we get taking maximal possible
value of $\alpha = 2(\gamma-3)$ that $$\lnorm{\ro}{3\gamma-6}\leq C.$$ Using this  in the already derived estimates for $\u,c$ and $\mu$ yields the result of Lemma \ref{basic}.
\end{proof}

Now, we will show that for $\gamma>6$ we can expect principally better regularity of the solutions, this is connected to the fact that in this case we can take according to
\refx{gammatri} $\alpha > \gamma$, so $p(\ro,c) \in L^{s}(\Omega),$ for some $s>2.$
\begin{lem}
For $\gamma>6$ we have for solutions to \refx{RK}-\refx{tokh} for any $1<p<+\infty$
\begin{align*}
 \norm{\u}_{W^{1,p}(\Omega)}+ \lnorm{\ro}{\infty}+\lnorm{\nabla c}{\infty}+\lnorm{\mu}{\infty}+\lnorm{\ro L\carka(c)}{\infty} \leq C_p,
\end{align*}
and $c\in[0,1]$ a.e. in $\Omega$.
\end{lem}

\begin{proof}

First, since $\gamma>6$ we observe certain smoothing effect of \refx{third} and \refx{system}. In what follows, we will repeatedly use  H\"{o}lder's inequality in the third equation and Lemma \ref{koncentr}.
Indeed, since $\ro\u\cdot\nabla c = -\mu,$ and $\ro\in L^{3\gamma-6}$, $\u\in L^{6}$ and $\nabla c \in L^{\frac{6\gamma-12}{\gamma}}$ we get $\mu\in L^{\frac{3\gamma-6}{\gamma}}$,
 $(\frac{1}{3\gamma-6}+\frac{1}{6} + \frac{\gamma}{6\gamma-12}>\frac{1}{2})$, and 
applying Lemma \ref{koncentr} $\nabla c \in L^{\gamma-2},$ which can be again plugged into the third equation in order to get $\mu\in L^{\frac{6\gamma-12}{\gamma+6}},$ 
and again $\nabla c \in L^{\frac{6\gamma-12}{12-\gamma}},$ at least for $6<\gamma<12,$ etc. This procedure can be repeated until $\nabla c \in L^{\infty}$, since there exists no 
reasonable solution to the following system of algebraic equations, where $P$ corresponds to the expected integrability of $\nabla c $, $Q$ to $\mu$ and $\frac{3Q(\gamma-2)}{3\gamma+Q-6}$ 
corresponds to $\ro\mu$  \begin{align}
\frac{1}{3\gamma-6}+\frac{1}{6}+\frac{1}{P}=\frac{1}{Q},\qquad P = \frac{3 Q (\gamma-2)}{3\gamma-6-Q\gamma-Q}.
\end{align}
 So we get that  $P\to+\infty$, and $Q\to \frac{6\gamma-12}{\gamma}>3$, id est for any $\gamma>6$ after finite number of such steps we have $\nabla c \in L^{\infty}(\Omega)$, 
$\mu\in L^{\frac{6\gamma-12}{\gamma}}(\Omega)$, and $\Delta c \in L^{\frac{6\gamma-12}{\gamma+2}}(\Omega)$. Thus, to summarize
\begin{align*}
 \norm{\u}_{W^{1,2}(\Omega)}+ \lnorm{\ro}{3\gamma-6}+\lnorm{\nabla c}{\infty}+\lnorm{\mu}{\frac{6\gamma-12}{\gamma}}+ \lnorm{\Delta c}{\frac{6\gamma-12}{\gamma+2}} 
+\lnorm{\ro L\carka(c)}{\frac{6\gamma-12}{\gamma+2}} \leq C.
\end{align*}

>From the last norm we can deduce $ c \in [0,1] \text{ a.e. on } \{\ro>0\}$. On the other hand, $c$ is continuous in $\Omega$ and on the set $ \{\ro=0\}$ it satisfies 
the Laplace equation, and therefore maximum principle. Thus,\footnote{More precisely, since $c$ is continuous,  the set $U = \{c\notin [0,1]\}$ is open. Considering any  
ball $B(r,x_0)\subset U$ we get that $\ro = 0$ a.e. in $B(r,x_0)$, hence $c$ satisfies $\Delta c = 0 $ in $B(r,x_0)$ and can not reach neither maximum, nor minimum within this ball.}
$ c \in [0,1] \text{ a.e. in }\Omega.$

Now, we will use the fact that we work with slip boundary condition, and thus we can deduce from the momentum equation the following relation for vorticity ($\omega=\curl\u$), see \cite{Mu04}
\begin{align*}
- \nu \Delta \omega = -\curl (\ro \u\cdot\nabla\u )-\curl\bigl(\Delta c  \nabla c\bigr)+\curl (\ro\ef)&\text{ in }\Omega,\\ 
 \omega\cdot\boldsymbol{\tau}_1 =- (2\chi_2-\frac{k}{\nu})\u\cdot\boldsymbol{\tau}_2&\text{ on }\partial\Omega,\\ 
  \omega\cdot\boldsymbol{\tau}_2 = (2\chi_1-\frac{k}{\nu})\u\cdot\boldsymbol{\tau}_2&\text{ on }\partial\Omega,\\
\sol \omega =0 &\text{ on }\partial\Omega,
\end{align*}  
where $\chi_n$ are the curvatures related to the vectors $\boldsymbol{\tau}_n.$ 
Note, that $\en\cdot(\nabla c\otimes\nabla c)\cdot \boldsymbol{\tau}_n=0$ on $\partial\Omega$, since $\left.\nabla c \cdot\en\right\rvert_{\partial\Omega} = 0.$
Now, we would like to show that 
$\lnorm{\nabla\omega}{p}\leq C,$ for some $p>1$. The analysis of the above system (the operator defined on the left hand side) can be found in \cite{MuPo14}.

First, we have $\left.\u\cdot \boldsymbol{\tau}\right\rvert _{\partial\Omega} \in W^{\frac{1}{2},2}(\partial\Omega)$, so we control $\omega$ on the boundary. Further,
 $\ro\ef\in L^{3\gamma-6}(\Omega)$, $\ro\u\cdot\nabla\u\in L^p(\Omega),$ for any $p\in\bigl(1,\frac{6\gamma-12}{4\gamma-6}\bigr)$, and finally
$$ \lnorm{\Delta c\nabla c}{\frac{6\gamma-12}{\gamma+2}}\leq\lnorm{\Delta c}{\frac{6\gamma-12}{\gamma+2}}\lnorm{\nabla c}{\infty}\leq C.$$
According to the fact that $1<\frac{3\gamma-6}{2\gamma-3}<\frac{6\gamma-12}{\gamma+2}$  even for all $\gamma>\frac{8}{3}$, we get 
$$\omega \in W^{1,\frac{3\gamma-6}{2\gamma-3}}(\Omega).$$

Now, we will proceed in the same manner as in \cite{MuPo06}, using the Helmholtz decomposition of the velocity field $\u=\curl \A+\nabla \phi$, define  $G= -(2\nu+\eta)\Delta \phi + p(\ro,c)$, observe
$$\nabla G = -\ro\u\cdot\nabla\u+\nu\Delta\curl \A+(\ro\ef +\Delta c \nabla c)$$
and show that $G\in W^{1,\frac{3\gamma-6}{2\gamma-3}}(\Omega)\hookrightarrow L^{\frac{3\gamma-6}{\gamma-1}}(\Omega).$ This further yields
$$\ro^\gamma\in L^{12/5}(\Omega),\quad \u\in W^{1, 12/5}(\Omega)\hookrightarrow L^{12}(\Omega)$$
 and after more iterations $\ro\in L^{\infty}(\Omega),$
$\u\in W^{1,p}(\Omega),$ $c\in  W^{1,p}(\Omega),$ for any $1<p<\infty$ as  well.
\end{proof}

\def\eps{\epsilon}
\def\ka{m}

\section{Approximation}
\label{aproxsec}

In this section we define a problem approximating the original one and prove the existence of the corresponding solutions.
We introduce $h=M/\abs{\Omega}$, $\eps>0$, a smooth  cut-off function $K(\ro)$
$$K(\ro) =\begin{cases}
  1,&\text{ for }\ro\leq \ka-1 \\
   \in(0, 1) ,&\text{ for }\ro\in (\ka-1,\ka)\\
    0,&\text{ for }\ro\geq \ka,
\end{cases}$$ 
a "regularized logarithm" which is a function $l_\eps\in C^{1}([0,\infty))$ which is bounded from below by $\log(\sqrt[t]{\eps})-1$ ($t>1$ will be specified later) and
$$ l_\eps(s)=\begin{cases}
  \log(s),&\text{ for }s\geq{\sqrt[t]{\eps}}, \\
  \text{convex,  non-decreasing}&\text{ for }s<{\sqrt[t]{\eps}},
\end{cases} $$
with\footnote{We can get such an function e.g. by replacing the logarithm by a suitable affine function for small arguments.} \begin{align}
0\leq &\,\sqrt[t]{\eps} \bigl(2 l\carka_\eps(s) +s l\carky_\eps(s)\bigr) \leq  C, \nonumber\\
0\leq &\, s \bigl(2 l\carka_\eps(s) +s l\carky_\eps(s)\bigr) \leq C \text{ for a.a. }s \in[0,\infty), \label{logar}
\end{align}where $C$ is independent of $\eps$; further we denote
the approximated free energy
$$ f_\eps(\ro,c) =\Gamma(\ro) + (ac+d)l_\eps(\ro)+  L_\eps(c) + b(c),$$
where we define $L_\eps(c)=\int\limits_{0}^c l_{\eps}(s)-l_{\eps}(1-s)\de{s}$ for $c\in[0,1]$, and then extend it to whole $\R$ as a convex function with  $\lnorm{L\carka_\eps}{\infty}\leq- C \log{\sqrt[t]{\eps}}$, $\Gamma(\ro) = \frac{\ro^{\gamma-1}}{\gamma-1}$
and approximated pressure
$$p_\eps(\ro,c)= P_b(\ro) +(ac+d)\int_0^{\ro} K(s) \partial_s{\bigl(s^{2}l\carka_\eps(s)\bigr)}\:\de{s}, $$ where
$P_b(\ro)=\int_0^{\ro}\gamma s^{\gamma-1}K(s)\:\de{s}.$

Our approximation problem then reads
\begin{align}
\eps\ro + \sol(K(\ro)\ro\u)=\eps\Delta\ro + \eps K(\ro)h, \label{epscont}
\end{align}

\vspace*{-0.8cm}

\begin{multline}
\frac{1}{2}\sol(K(\ro)\ro \u\otimes\u)+\frac{1}{2}K(\ro)\ro\u\cdot\nabla\u -\nu\Delta\u-(\nu+\eta)\nabla(\sol\u)+
\nabla p _\eps(\ro)\\= K(\ro)\ro \ef+\sol\bigl(\nabla c \otimes\nabla c -
 \frac{\abs{\nabla c}^2}{2}\Id\bigr) - a\nabla c \int_{0}^{\ro}s K\carka(s)\de{s}, \label{MEaprox}
\end{multline}

\vspace*{-0.8cm}

\begin{align}
 K(\ro)\ro\u \cdot\nabla c =& -\mu,\\
 K(\ro)\ro\mu = -\Delta c + &K(\ro) \ro\derivuj{f_\eps}{c} +  \epsilon\ro K(\ro)L\carka(c).\label{epsconc}
\end{align}
Moreover, we supply the equations with additional boundary condition \begin{equation}
 \nabla \ro\cdot\en = \boldsymbol{0}.\label{BCeps}
\end{equation}

\begin{prop} \label{paprox}
Let $\eps>0.$ Suppose that the assumptions of Theorem \ref{main} are satisfied, then there exists at least one solution $\ro_\eps,\u_\eps, \mu_\eps, c_\eps$ to the system
 \refx{epscont} - \refx{epsconc} with \refx{imper} - \refx{conserv}, \refx{BCeps}. Moreover,
we have with $1<q<+\infty$ the following estimates independent of $\eps$
\begin{equation}
\begin{split}
\lnorm{\ro_\eps}{\infty} + \norm{\u_\eps}_{W^{1,q}(\Omega)} + \eps\lnorm{\nabla\ro}{2}^{2} + \eps\lnorm{(K(\ro)\ro)^{1/2}L\carka(c)}{2} \leq C(\ka),\\
\lnorm{p_\eps}{2} +\lnorm{\mu}{2}+ \lnorm{\nabla c}{\frac{6\gamma}{3+\gamma}} + \norm{\u_\eps}_{W^{1,2}(\Omega)} + 
\lnorm{\bigl(K(\ro)\ro\bigr)^{\frac{\gamma+1}{2\gamma}} L_\eps\carka(c)}{\frac{2\gamma}{\gamma+1}} \leq C. \label{epsest} 
\end{split}
\end{equation}
\end{prop}

\begin{proof}
 The existence of solutions for the approximative system will be deduced by means of the following variant of the Leray-Schauder fixed point theorem, see e.g. \cite {Ev}.
 \begin{LSFPT}
 Let $X$ be a Banach space, and $\Te$ a continuous and compact mapping $\Te:\:X \mapsto X$, such that the possible fixed points
  $x = t \Te x$, $0 \leq t \leq 1$
 are bounded in $X$. Then $\Te$ possesses a fixed point.
 \end{LSFPT}
 We define spaces
 \begin{align*}
 \mathbf{M}_q= &\left\{\mathbf{w}\in \bigl(W^{2,q}(\Omega),\R^3\bigr),\mathbf{w}\cdot\en=\boldsymbol{0} \text{ on }\partial\Omega \right\},\\
 M_q=&\left\{m\in W^{2,q}(\Omega),\nabla m\cdot\en=0 \text{ on }\partial\Omega \right\},\\
 N_q=&\left\{z\in W^{3,q}(\Omega),\:\nabla z\cdot\en=0 \text{ on }\partial\Omega \right\},
 \end{align*}
and search for $\ro\in M_p$, $\u\in\mathbf{M}_p,$ and $c\in N_p$, with $1\leq p<\infty,$  see \cite{NoSt04, Feir04} for similar considerations for the Navier-Stokes system.  Let us first concentrate on the continuity equation. 
\begin{lem}
The solution operator $\mathcal{S}_1(\u)=\ro$ of the problem
\begin{align*}
\eps\ro + \sol(K(\ro)\ro\u)=\eps\Delta\ro + \eps K(\ro)h\text{ in }\Omega,\\
\nabla\ro\cdot\en = \boldsymbol{0}\text{ on }\partial \Omega
\end{align*}
is for $p>3$ a continuous compact operator from $\mathbf{M}_p$ to $W^{2,p}.$ Moreover,
$\ro\geq 0,$ $\inte{\ro\:}\leq M,$ and $$\norm{\ro}_{W^{2,p}}\leq C(\ka,\epsilon)(1+\norm{\u}_{W^{1,p}(\Omega)}).$$
\end{lem}
\begin{proof}
The proof can be found in \cite{MuPo09}, see also \cite{MuPo06, Lasi14}. We recall here only the idea how to obtain the estimates.
First, considering the subset $\{\ro< 0\}\subset\Omega$ we get $\ro\geq0$ a.e. in $\Omega$, then integrating the approximate continuity equation over $\Omega$ yields $$ \inte{\ro}= h\inte{K(\ro)},$$ so $K(\ro) = 1 $ a.e. in $\Omega.$
Further, testing the continuity equation by $\ro$ yields
\begin{equation*}
\begin{split}
\eps\inte{\ro^{2}+\abs{\nabla \ro}^{2}} -\eps\inte{K(\ro)\ro h }\leq & -\inte{\ro \sol\bigl(K(\ro) \ro\u\bigr) } = \inte{\u\cdot\nabla\ro K(\ro)\ro}\\
 =& \inte{ \u\cdot\nabla\Bigl( \int_0^\ro K(s)s \de{s}\Bigr)} \leq C\lnorm{K(\ro)\ro^{2}}{2}\lnorm{\nabla\u}{2},
\end{split}
\end{equation*}
since the last term on the left hand side can be easily bounded we get
\begin{equation}
\eps\lnorm{\nabla\ro}{2}^2\leq C\bigl(1+\lnorm{K(\ro)\ro^{2}}{2}\lnorm{\nabla\u}{2}\bigr).\label{kontin}
\end{equation} 
\end{proof}

Similarly, for the last two equations we have 
\begin{lem}
The solution operator $\mathcal{S}_2(\u)= c$ of the problem
\begin{align*}
-\mu  = & \:K(\ro)\ro\nabla c\cdot \u,\\
 -\Delta c +K(\ro)\ro\epsilon L\carka(c) =&\:K(\ro)\ro\mu - K(\ro)\ro\derivuj{f_\eps}{c}\text{ in }\Omega\text{, where }\ro = \mathcal{S}_1(\u), \\
\nabla c\cdot\en  = & \:\boldsymbol{0}\text{ on }\partial \Omega
\end{align*}
is for $p>3$ a well-defined compact operator from $\mathbf{M}_p$ to $ N_p.$ Moreover,
$c\in [0,1],$ and $$\norm{c}_{W^{2,\frac{2q}{2+q}}} + \lnorm{\nabla c}{\frac{6q}{6+q}}\leq C (\lnorm{\mu}{2} \lnorm{K(\ro)\ro}{q}+1).$$
\end{lem}
\begin{proof}
The proof is quite similar to the previous Lemma, for the estimates we proceed analogously as in the proof of Lemma \ref{koncentr}. 
For constructing the solution we use again Schauder fixed point theorem. We consider for fixed $\ro\in M_p$ the mapping $c\to z$ defined as a solution operator to the problem
\begin{align}
-\mu =& \:K(\ro)\ro\nabla c\cdot \u,\nonumber\\
 -\Delta z +K(\ro)\ro\epsilon L\carka(z) =&\:K(\ro)\ro\mu - K(\ro)\ro\derivuj{f_\eps}{c}(\ro,c),\label{neumann}\\
\nabla z\cdot\en = & \:\boldsymbol{0}\text{ on }\partial \Omega.\nonumber
\end{align}
The second  equation is for $\epsilon>0$ strictly elliptic,  furthermore, its right hand side belongs 
to $W^{1,p}(\Omega)\hookrightarrow L^{\infty}(\Omega)$, in particular we  deduce that $K(\ro) \ro L\carka(c)\in L^\infty(\Omega)$ and $\nabla c\in W^{1,\infty}(\Omega)$,
 hence since $\ro$ and $z$ are continuous we find $K(\ro) \ro L\carky(z)\in L^\infty(\Omega)$, put the corresponding term to the right hand side, observe that in fact 
$K(\ro)\ro\mu - K(\ro)\ro\derivuj{f_\eps}{c}-K(\ro)\ro\epsilon L\carka(z)\in W^{1,p}(\Omega)$, and so $z\in W^{3,p}(\Omega)$, and the corresponding mapping is compact by the same reasons as above.
In order to get the desired estimate independent of $\eps$, we test the equation
$-\Delta c + K(\ro)\ro\bigl( L\carka_\eps(c) + \eps L\carka(c)  \bigr) = K(\ro)\ro\mu - K(\ro)\ro 
\bigl( b\carka(c) + a l_\eps(\ro) \bigr)$ by $F(L\carka_\eps(c) +\eps L\carka(c))$, where $F$ is increasing function such that  $\ksi F(\ksi)\sim\abs{\ksi}^{\beta+1}$,
 $\abs{F}\sim \abs{\ksi}^{\beta},$ for some $\beta>0.$ Note that $L\carka_\eps(c) +\eps L\carka(c)$ is non-decreasing function, so we get
\begin{equation}
\begin{split}
&\eps\lnorm{(K(\ro)\ro)^{1/(\beta+1)}L\carka(c)}{\beta+1} ^{\beta+1} + \lnorm{(K(\ro)\ro)^{1/(\beta+1)}L_\eps\carka(c)}{\beta+1} ^{\beta+1}\\
 &\qquad\leq\inte{\abs{\nabla c}^2 F\carka\bigl(L\carka_\eps(c) +\eps L\carka(c)\bigr)\bigl( L\carky_\eps(c) +\eps L\carky(c)\bigr) +K(\ro)\ro \bigl(L\carka_\eps(c)
+\eps L\carka(c)\bigr)F\bigl(L\carka_\eps(c) +\eps L\carka(c)\bigr) }\\
 &\qquad\leq\abs{ 
\inte{K(\ro)\ro\mu F\bigl(L\carka_\eps(c) +\eps L\carka(c)\bigr)}} +\abs{\inte{  K(\ro)\ro \bigl( b\carka(c) + a l_\eps(\ro) \bigr) F\bigl(L\carka_\eps(c) +\eps L\carka(c)\bigr) }},
\end{split}
\end{equation} 
hence \begin{multline}\eps\lnorm{(K(\ro)\ro)^{\frac{1}{\beta+1}}L\carka(c)}{\beta+1}  + \lnorm{(K(\ro)\ro)^{\frac{1}{\beta+1}}L_\eps\carka(c)}{\beta+1}\\
\leq C(\lnorm{\mu}{p} \lnorm{K(\ro)\ro}{\frac{p}{p-(\beta+1)}}+1)\label{Leps} \end{multline}
and using the classical elliptic estimates we obtain as in Lemma \ref{koncentr}
$$\norm{c}_{W^{2,\frac{pq}{p+q}}} + \lnorm{\nabla c}{\frac{3pq}{3p+3q-pq}}\leq C (\lnorm{\mu}{p} \lnorm{K(\ro)\ro}{q}+1), $$
especially for $p=2$,
$$\norm{c}_{W^{2,\frac{2q}{2+q}}} + \lnorm{\nabla c}{\frac{6q}{6+q}}\leq C (\lnorm{\mu}{2} \lnorm{K(\ro)\ro}{q}+1).$$
The issue of existence of the solutions to \refx{neumann} requires some comments. The function $L\carka(\cdot)$ is singular and it keeps the value of $z$ in the interval $[0,1]$. Thus, we approximate  \refx{neumann} by its regularization substituting $L\carka(\cdot)$ by  $L\carka_\delta(\cdot)$ which is obtained in the same manner as for $f_\eps$ in \refx{logar}. The estimates are the same, there is no problem to pass to the limit $\delta\to0,$ since we control the second derivatives of $z.$ Hence, we ensure that $z\in[0,1]$ as well.
\end{proof}

Finally, we define the solution operator $\mathcal{T}:\mathbf{M}_p \to \mathbf{M}_p$, $\mathcal{T}(\u) = \mathbf{w}$ of the problem
\begin{align}
 -\nu\Delta\mathbf{w}-(\nu+\eta)\nabla(\sol\mathbf{w})=&
 -\frac{1}{2}\sol(K(\ro)\ro \u\otimes\u)-\frac{1}{2}K(\ro)\ro\u\cdot\nabla\u -\nabla P _\epsil(\ro )+ K(\ro)\ro \ef\nonumber\\ 
  &\qquad\quad\qquad+\sol\bigl(\nabla c\otimes\nabla c-\frac{\abs{\nabla c}^2}{2}\Id\bigr) - a\nabla c \int_{0}^{\ro}s K\carka(s)\de{s},\label{wellipt}\\
  &\qquad\qquad\qquad\text{ where }\ro = \mathcal{S}_1(\u)\text{, and } c= \mathcal{S}_2(\u)\nonumber
  \end{align}
equipped with the boundary condition
  \begin{gather*}
 \mathbf{w}\cdot\en =0,\quad
    \en\cdot \T(\mathbf{w})\cdot \boldsymbol{\tau_n} +k\mathbf{w} \cdot\boldsymbol{\tau} =0,\text{ on }\partial\Omega.
  \end{gather*}
\begin{lem}
$\mathcal{T}$ is continuous and compact operator from $\mathbf{M}_p$ to $\mathbf{M}_p$ for $p>3$.
\end{lem}
\begin{proof}[Idea of the proof.]
It is again one more time strictly elliptic system with  right hand side which belongs at least to the $L^p(\Omega)$ and it contains at most first order derivatives of $\ro,\u$ and 
at most second order derivatives of $c$, see e.g. \cite{NoSt04} for similar considerations.
\end{proof}

Finally, we will verify that all possible solutions of $t \mathcal{T}(\mathbf{\u})=\mathbf{\u}$, for $t\in[0,1]$ are bounded in $\mathbf{M}_p$ independently of $t$.
Testing the momentum equation by $\u$ yields (with usage of the last equation tested by $\u\cdot\nabla c$)
\begin{multline}t\inte{\u\cdot\nabla p_\eps(\ro,c)\:}+ \inte{\nu\abs{\nabla\u}^2 + (\nu+\eta)\abs{\sol\u}^2} + 
\int\limits_{\partial\Omega} k (\abs{\u_1\cdot \boldsymbol{\tau}_1}^2+ \abs{\u_2\cdot \boldsymbol{\tau}_2}^2)\de{S}\\
\leq t \inte{K(\ro)\ro\ef\cdot\u}+t\inte{\Bigl(K(\ro)\ro\mu\nabla c\cdot \u-K(\ro) \ro \derivuj{f_\eps}{c}\nabla c \cdot \u\Bigr)}\\
+t\inte{\Bigl(-\eps\ro K(\ro)L\carka(c)\nabla c\cdot\u -a\nabla c\cdot\u \int_{0}^{\ro}s K\carka(s)\de{s}\Bigr)},\label{base}
\end{multline}
from the equation for concentration we get
$\inte{\mu^2}=-\inte{K(\ro)\ro\mu\u\cdot\nabla c\,}.$
Next,
\begin{multline}
\inte{\u\cdot\nabla p_\eps(\ro,c)\:} = \inte{\u\cdot \nabla \ro\Bigl( \gamma \ro^{\gamma-1} K(\ro) + (ac+d)K(\ro) \frac{\de{}}{\de{\ro}} \bigl(\ro^{2}l\carka_\eps(\ro) \bigr) \Bigr)}\\ 
+ \inte{a\nabla c\cdot\u \int_{0}^{\ro} K(s) \frac{\de{}}{\de{s}}\bigl(s^{2}l\carka_\eps(s)\bigr)\de{s}\:}\label{tlak}
\end{multline}   
and
\begin{equation}
\begin{split}
&\inte{K(\ro)\ro \u\cdot\nabla c \derivuj{f_\eps}{c}(\ro,c)}=\inte{K(\ro)\ro \u\cdot\nabla\bigl(f_\eps + \ro\derivuj{f_\eps}{\ro}\bigr)} - \inte{ K(\ro)\ro^{2} a l_\eps\carka(\ro) \u\cdot\nabla c\:}\\
& \qqquad-\inte{K(\ro)\ro \u\cdot\nabla\ro\Bigl( 2\Gamma\carka(\ro) + \ro \Gamma\carky(\ro)+(ac+d)l_\eps\carka(\ro) +(ac+d)\frac{\de{}}{\de{\ro}}\bigl(\ro l_\eps\carka(\ro)\bigr) \Bigr) }\\
& \qqquad\qqquad\qqquad=\inte{\sol\Bigl(K(\ro)\ro \u\bigl(f_\eps + \ro\derivuj{f_\eps}{\ro}\bigr)  \Bigr)} - \inte{\sol\bigl(K(\ro)\ro\u\bigr) \bigl(f_\eps + \ro\derivuj{f_\eps}{\ro}\bigr)  } \label{efko} \\
& \qqquad\qqquad\qqquad\qqquad - \inte{a\nabla c\cdot\u \int_{0}^{\ro} \frac{\de{}}{\de{s}} \bigl( K(s) s^{2}l\carka_\eps(s)\bigr)\de{s}\:} \\
& \qqquad\qqquad\qqquad\qqquad\qqquad + \inte{K(\ro)\u\cdot\nabla \ro 
\frac{\de{}}{\de{\ro}} \bigl( \ro^{2}\Gamma\carka(\ro) + \ro^{2}l\carka_\eps(\ro)(ac+d) \bigr)}.
\end{split}
\end{equation}
The first term on the right hand side of \refx{efko} can be eliminated by boundary conditions, for the second one we will use the continuity equation 
(tested by $(f_\eps+\ro\derivuj{f_\eps}{\ro})$). The third one  cancels out with the last terms of \refx{base} and \refx{tlak}, and the last term is exactly
 the same as the main part of \refx{tlak}.

Thus,  summing up the resulting inequalities with appropriate powers of $t$ we get 
\begin{multline}
\norm{\u}_{\sob}^2+t\lnorm{\mu}{2}^2  +t\inte{\epsilon\ro\bigl(f_\eps+\ro\derivuj{f_\eps}{\ro}\bigr)}
 \leq t\inte{K(\ro)\ro\ef\cdot\u} - t\intek{\eps\ro K(\ro) L\carka(c) \u\cdot\nabla c }\\ 
  +t\inte{\bigl(\epsilon\Delta\ro + \epsilon K(\ro)h \bigr)\bigl(f_\eps +\ro\derivuj{f_\eps}{\ro}\bigr)   },\label{energie}
\end{multline}
where we have introduced the notation $\Omega_K = \{K(\ro)\ro>0\}\cap\Omega.$
We  estimate the $\eps$-dependent terms.
\begin{equation}
\begin{split}
&-\eps\intek{K(\ro)\ro L\carka(c)\nabla c\cdot\u}=-\intek{\eps K(\ro)\ro\u\cdot\nabla \bigl(L(c)\bigr)}\\&\qqquad = \int\limits_{\partial\Omega_K} \eps \ro K(\ro) \u\cdot\en L(c)\de{S} 
+ \intek{\eps\sol(K(\ro)\ro\u)L(c) }\\&\qqquad
\leq \sqrt{\eps}\lnorm{\nabla\ro}{2}\lknorm{\ro K\carka(\ro) + K(\ro)}{3}\lnorm{\u}{6}\sqrt[4]{\eps}\lknorm{L(c)}{\infty}\sqrt[4]{\eps}\\&\qqquad\qqquad\qqquad\qqquad
+\lnorm{K(\ro)\ro}{2}\lnorm{\nabla\u}{2}\sqrt[4]{\eps}\lknorm{L(c)}{\infty} \eps^{3/4}
\end{split}
\end{equation}
and we can use $\sqrt[4]{\eps}\lknorm{L(c)}{\infty} \leq C( \sqrt[4]{\eps}\lnorm{ (K(\ro)\ro)^{1/(\beta+1)}  L\carka(c)}{\beta+1}^{1/4} +1)$ with estimate \refx{Leps}.

Concerning the term $\eps K(\ro)(f_\eps +\ro\derivuj{f_\eps}{\ro})$ we have three parts. First, $\Gamma_\eps(\ro) + \ro\Gamma_\eps\carka(\ro)$, 
which can be bounded by the corresponding term on the left hand side, see $\eps(\ro f_\eps).$ Second, $(ac+d)\bigl(1+l_\eps(\ro)\bigr)$, which has good sign on 
the set $\{\ro\leq \mathrm{e}^{-1}\}$ and is bounded by $c(1+\ro) K(\ro)$ on $\{\ro> \mathrm{e}^{-1}\}.$ The third term $\eps \bigl(L_\eps(c)+ b(c)\bigr)$ can be bounded 
according to our definition of $l_\eps$. For $\Delta\ro(f_\eps +\ro\derivuj{f_\eps}{\ro})$ we have
\begin{multline}
\inte{\eps \Delta\ro(f_\eps +\ro\derivuj{f_\eps}{\ro})} =  - \eps\inte{ \abs{\nabla\ro}^2 \Bigl(\gamma\ro^{\gamma-2} +(ac+d) \bigl(\ro l_\eps\carky(\ro)+2l_\eps\carka(\ro) \bigr)\Bigr) } \\
 -\eps\inte{ \nabla\ro\cdot\nabla c \bigl( l_\eps\carka(\ro)\ro + b\carka(c) + L\carka_\eps(c)\bigr)}.
\end{multline}
The first integral has a good sign. Indeed, $\ro^{\gamma-2}$ as well as $\ro l_\eps\carky(\ro)+2l_\eps\carka(\ro)$ are non-negative (for large arguments we have 
$\ro l_\eps\carky(\ro)+2l_\eps\carka(\ro) = \frac{1}{\ro}\geq 0 $, in the other case the conclusion is obtained from the fact that for small arguments  $l_\eps$ is increasing and convex, 
see \refx{logar}). Note also that  we have \refx{kontin} and for $\eps>0$ the approximated version of \refx{system} yields $c\in[0,1]$ as soon as we control $\lnorm{K(\ro)\ro}{2\gamma}\lnorm{\mu}{2}$. 
For the rest we get denoting $V_\eps(c) =  b\carka(c) +L_\eps\carka(c)$  and $U = \norm{\u}_{\sob}^2+t\lnorm{\mu}{2}^2$ that
\begin{multline}
t\inte{\eps \abs{\nabla{\ro}}\abs{\nabla c}\abs{ l_\eps\carka(\ro)\ro +V_\eps(c)}   }\leq t \sqrt[4]{\eps} \lnorm{\sqrt{\eps}\nabla\ro}{2}\lnorm{\nabla c}{2}\sqrt[4]{\eps}\bigl(
\lnorm{ l_\eps\carka(\ro)\ro+V_\eps(c)}{\infty}\bigr)\\
\leq t\sqrt[4]{\eps}C\lnorm{K(\ro)\ro^2}{2}^{1/2}\lnorm{\nabla\u}{2}^{1/2}\lnorm{K(\ro)\ro}{\gamma}\lnorm{\mu}{2} \\ 
\leq t  \sqrt[4]{\eps}C\lnorm{K(\ro)\ro^2}{2}^{1/2}\lnorm{K(\ro)\ro}{\gamma}U^{3/4},
\end{multline}
where we have used the choice of $l_\eps$ which guarantees that $\sqrt[4]{\eps}
\lnorm{ l_\eps\carka(\ro)\ro}{\infty} +\sqrt[4]{\eps} \lnorm{V_\eps(c)}{\infty}\bigr)\leq C$ independently of $\eps$.
 Thus,
 \begin{multline*}
 U\leq C \:\Bigl(1+\lnorm{K(\ro)\ro}{6/5}^2 + \eps\lnorm{K(\ro)\ro}{2} U^{3/4} + \sqrt[4]{\eps}\lnorm{K(\ro)\ro^2}{2}^{1/2}\lnorm{K(\ro)\ro}{3}U^{3/4}  \Bigr).
 \end{multline*}
However $\lnorm{K(\ro)\ro}{6}$ is definitely finite so finally,
$$U\leq C\Bigl(1+\lnorm{K(\ro)\ro}{6/5}^2),\qqquad\lnorm{\nabla c}{\frac{6q}{6+q}}\leq C(\lnorm{K(\ro)\ro}{q}\lnorm{\mu}{2} +1). $$
The Bogowskii estimates go along exactly the same lines as in the a priori approach as soon as we observe that
\begin{multline}
\inte{a\u\cdot\nabla c \int_0^{\ro}t K\carka(t)\de{t} \Phib }\\ 
\leq C\lnorm{\u}{6}\lnorm{\nabla c}{\frac{6\gamma}{3+\gamma}}\lnorm{K(\ro)\ro}{\frac{6\gamma}{3\gamma-3}}\lnorm{\Phib}{6}\leq 
 \lnorm{K(\ro)\ro}{2\gamma}^{1+\gamma+\frac{2\gamma} {3(2\gamma-1)}},
\end{multline}
so $$\lnorm{\ro_\eps}{3\gamma-6}\leq C,$$
independently of $\ka$ and  $\eps$. Furthermore, by the same iteration process applied on the last two equations of \refx{RK}-\refx{system} as in the a priori approach we can deduce that
\begin{equation} \lnorm{\mu_\eps}{\frac{6\gamma}{\gamma+3}} + \lnorm{\nabla c_\eps}{\infty}+ \lnorm{\Delta c_\eps}{\frac{6\gamma}{6+\gamma}}\leq C .\label{estcon}\end{equation}
Having these estimates in hands and noting that $\lnorm{K(\ro)\ro}{\infty}\leq C(m)$ we can apply the elliptic theory on the equation \refx{wellipt} and get the estimate of fixed points of $\Te$ in $\mathbf{M}_p$. 
This completes the proof of Proposition \ref{paprox}.
\end{proof}

\section{Artificial diffusion limit}

The last section of the paper is dedicated to the proof of convergence of the constructed approximative solutions to a weak solution to the original system. As usual the key part
is related to the proof of the strong convergence of the densities.

Thanks to the estimates \refx{epsest}, \refx{estcon} we  extract from the family $(\ro_\epsilon, \u_\epsilon, \mu_\epsilon, c_\epsilon)$ subsequences which converge in the corresponding spaces 
as $\epsilon\to 0+.$ Namely,\footnote{We denote a weak limit of nonlinear expressions $\{ w_{\epsilon}(\ro_\eps,\u_\eps,\mu_\eps, c_\eps)\}$ by $\overline{w(\ro,\u,\mu, c)}.$}
\begin{align*}
\u_\eps \weak\: &\u \text{ in }W^{1,q}(\Omega),&\u_\eps \to\: &\u \text{ in }L^{\infty}(\Omega), \\
\ro_\eps \weaks\:&\ro \text{ in }L^{\infty}(\Omega),&
p_\eps(\ro,c) \weaks\:& \overline{p(\ro,c)}\text{ in }L^{\infty}(\Omega), \\
K(\ro_\eps)\ro_\eps \weaks\:&\overline{K(\ro)\ro} \text{ in }L^{\infty}(\Omega),&
K(\ro_\eps)\weaks\:&\overline{K(\ro)} \text{ in }L^{\infty}(\Omega),\\
\int_{0}^{\ro_\eps}t K\carka(t)\de{t} \weaks\:&\overline{\int_{0}^{\ro}t K\carka(t)\de{t}} \text{ in }L^{\infty}(\Omega),&
h_\eps(\ro_\eps)\weaks\:& \overline{h_\eps(\ro)} \text{ in }L^{\infty}(\Omega),\\
c_\eps \weak\: &c \text{ in }W^{2,2}(\Omega),&\nabla c_\eps \to\: &\nabla c \text{ in }L^{6}(\Omega), \\
\eps K(\ro_\eps)\ro_\eps L\carka(c_\eps) \weak\:& \overline{\eps K(\ro)\ro }L\carka(c) \text{ in }L^{2}(\Omega),&
 K(\ro_\eps)\ro_\eps L_\eps\carka(c_\eps) \weak\:& \overline{K(\ro)\ro }L\carka(c) \text{ in }L^{12/7}(\Omega), \\
\mu_\eps \weak\:&\mu \text{ in }L^{2}(\Omega),&
K(\ro_\eps)\ro_\eps\mu_\eps\weak\:& \overline{K(\ro)\ro\mu}  \text{ in }L^{2}(\Omega).
\end{align*}
  Thus, we get
 \begin{align*}
\sol(\overline{K(\ro)\ro}\u)= &0, \\
\frac{1}{2}\sol(\overline{K(\ro)\ro} \u\otimes\u)+\frac{1}{2}\overline{K(\ro)\ro}&\:\u\cdot\nabla\u -\nu\Delta\u-(\nu+\eta)\nabla(\sol\u)+\nabla\overline{ p _\eps(\ro,c)}\\
&= \overline{K(\ro)\ro} \ef +\sol\bigl(\nabla c\otimes\nabla c-\frac{\abs{\nabla c}^2}{2}\Id\bigr) - a\nabla c \overline{\int_{0}^{\ro}t K\carka(t)\de{t}},\\
  \overline{K(\ro)\ro}\:\u\cdot\nabla c=& -\mu,\\
  \overline{K(\ro)\ro\mu} =& -\Delta c +\overline{h_\eps(\ro)} +\overline{K(\ro) \ro} \bigl(L\carka(c) + b\carka(c)\bigr) + \overline{\eps K(\ro)\ro } L\carka(c),\\
  \u\cdot\en =\: & 0,\quad\quad \nabla c\cdot\en = 0,\\
  \en\cdot \T(c,\u)\cdot\boldsymbol{\tau}_n&+k\u\cdot\boldsymbol{\tau}_n =0,\text{ on }\partial\Omega,
  \end{align*}
 where $ h_\eps(s) = a s \cdot l_\eps(s).$  Note that, due to the high regularity we have the pointwise convergence of concentrations. In order to show the pointwise convergence
 of  densities we need to investigate the momentum equation, especially its potential part defining the effective viscous flux. Let us decompose the velocity field 
$\u$ using Helmholtz decomposition, i.e.
 $$\u = \nabla\mfi + \curl \A, $$
 where \begin{align*}
 \curl\curl \A = \curl \u =\: &\omega\text{ in }\Omega,& &\Delta\mfi = \curl \u \text{ in }\Omega,\\
 \sol\curl\A =\: & 0\text{ in }\Omega,&\text{and}\qquad\qqquad & \nabla\mfi\cdot\en = 0\text{ on }\partial\Omega,\\
 \curl\A\cdot\en =\: & 0\text{ on }\partial\Omega,& &\inte{\mfi\:} =  0.
 \end{align*}
For the stream function $\A$ we have good estimates
$$ \lnorm{\nabla\curl\A}{q}\leq \lnorm{\omega}{q},\qquad   \lnorm{\nabla^2\curl\A}{q}\leq \norm{\omega}_{W^{1,q}(\Omega)},$$
 and since $\omega$ solves
\begin{align}
-\nu \Delta \omega = -\curl \bigl(\overline{K(\ro)\ro} \u\cdot\nabla\u\bigr)+\curl \bigl(\overline{K(\ro)\ro} \ef\bigr)
 +\curl\Bigl(\Delta c \nabla c - a\nabla c \overline{\int_{0}^{\ro}t K\carka(t)\de{t}}\Bigr),
\end{align}
we have $\norm{\omega}_{W^{1,q}(\Omega)}\leq C(\ka).$ Similarly we also decompose the approximative velocity field as $\u_\eps = \nabla\mfi_\eps + \curl \A_\eps $ and deduce due to the slip boundary 
conditions the following problem for vorticity
\begin{equation*}
\begin{split}
-\nu \Delta \omega_\eps =& \curl \Bigl(K(\ro_\eps)\ro_\eps \ef-K(\ro_\eps)\ro_\eps \u_\eps\cdot\nabla\u_\eps -\frac{1}{2}\eps h K(\ro_\eps)\u_\eps   \Bigr.\\& 
\underbrace{\qquad\qqquad\quad
\Bigl.+\eps \frac{1}{2}\ro_\eps\u_\eps+\Delta c_\eps \nabla c_\eps - a\nabla c_\eps \int_{0}^{\ro_\eps}t K\carka(t)\de{t}\Bigr) }\limits_{=:\Ha_1} 
\underbrace{-\curl \bigl(\frac{1}{2}\eps\Delta\ro_\eps \u_\eps  \bigr)}\limits_{=:\Ha_2} =: \Ha_1 + \Ha_2, \\&\qqquad
 \omega_\eps\cdot\boldsymbol{\tau}_1 =- (2\chi_2-\frac{k}{\nu})\u_\eps\cdot\boldsymbol{\tau}_2\text{ on }\partial\Omega,\\ &\qqquad
  \omega_\eps\cdot\boldsymbol{\tau}_2 = (2\chi_1-\frac{k}{\nu})\u_\eps\cdot\boldsymbol{\tau}_2\text{ on }\partial\Omega,\\
  &\qquad\quad\sol\omega_\eps=0\text{ on }\partial\Omega.
\end{split}
\end{equation*}
The structure of the right hand side of the first equation enables us to consider $\omega_\eps =\omega_\eps^0+\omega_\eps^1+\omega_\eps^2 $ as a sum of solutions to three particular systems, namely
\begin{align}
\nu\Delta\omega_\eps^0 =& 0, & \nu\Delta\omega_\eps^1 =&\Ha_1, &\nu\Delta\omega_\eps^2 =\Ha_2 \:\text{ in }\Omega,\nonumber\\
 \omega_\eps^0\cdot\boldsymbol{\tau}_1 =&- (2\chi_2-\frac{k}{\nu})\u_\eps\cdot\boldsymbol{\tau}_2 ,&\omega_\eps^1\cdot\boldsymbol{\tau}_1 = & 0,
& \omega_\eps^2\cdot\boldsymbol{\tau}_1=0 \:\text{ on }\partial\Omega,\nonumber\\
  \omega_\eps^0\cdot\boldsymbol{\tau}_2 =&- (2\chi_1-\frac{k}{\nu})\u_\eps\cdot\boldsymbol{\tau}_1 ,&\omega_\eps^1\cdot\boldsymbol{\tau}_2 = & 0,
& \omega_\eps^2\cdot\boldsymbol{\tau}_2=0 \:\text{ on }\partial\Omega,\nonumber\\
  \sol\omega_\eps^0 = & 0, &   \sol\omega_\eps^1 = & 0,&   \sol\omega_\eps^2 = 0  \:\text{ on }\partial\Omega.\label{omeg}
\end{align}

\begin{lem}\label{vort}
If the vorticity $\omega_\eps =\omega_\eps^0+\omega_\eps^1+\omega_\eps^2 $ solves \refx{omeg}, then we have
\begin{align}
\norm{\omega_\eps^0}_{W^{1,q}(\Omega)} + \norm{\omega_\eps^1}_{W^{1,q}(\Omega)}\leq & \: C\bigl( 1 + \ka^{1+ \gamma (\frac{4}{3}-\frac{2}{q})}  \bigr),\:\text{ for } q\in \left[2,\frac{6\gamma}{6+\gamma}\right],\\
\lnorm{\omega_\eps^2}{q}\leq & \: C(\ka)\eps^{1/2}, \:\text{ for } q\in[1,2].
\end{align}
\end{lem}
\begin{proof}
Following closely the corresponding considerations in \cite{MuPo08}, we deduce that (see \cite{MuPo14})
$$\norm{\omega_\eps^1}_{W^{1,q}(\Omega)}\leq C \norm{\Ha_1}_{W^{-1,q}(\Omega)}\text{ and } \norm{\omega_\eps^0}_{W^{1,q}(\Omega)}\leq C\norm{\u_\eps}_{W^{1,q}(\Omega)},$$
but according to \refx{estcon} for any $q$ such that $2\leq q\leq \frac{6\gamma}{6+\gamma}$ we have $$\norm{\u_\eps}_{W^{1,q}(\Omega)}\leq C ( \lnorm{p_\eps(\ro_\eps,c_\eps)}{q} 
+\lnorm{ \Delta c_\eps \nabla c_\eps}{q} +1)\leq C\bigl(1+ \ka^{\gamma(1-\frac{2}{q})}\bigr),$$
so we  concentrate on $\omega_\eps^1.$ If we denote $r=\frac{6\gamma}{6+\gamma}>3$ we get by interpolation
\begin{equation}
\begin{split}
\norm{\omega_\eps^1}_{W^{1,q}(\Omega)}\leq &C\bigl( 1 + \lnorm{K(\ro_\eps)\ro_\eps \u_\eps\nabla\u_\eps}{q} +\lnorm{\Delta c_\eps\nabla c_\eps}{q} \bigr)\\
\leq &C\bigl( 1 + \ka\lnorm{ \u_\eps}{\infty}\lnorm{\nabla\u_\eps}{q}  \bigr)\\
\leq &C\bigl( 1 + \ka\lnorm{ \u_\eps}{6}^{\frac{2r-6}{3r-6}}\lnorm{\nabla\u_\eps}{r}^{\frac{r}{3r-6}}\lnorm{\nabla\u_\eps}{2}^{\frac{2(r-q)}{q(r-2)}}\lnorm{\nabla\u_\eps}{r}^{\frac{r(q-2)}{q(r-2)}}  \bigr)\\
\leq &C\bigl( 1 + \ka\lnorm{\nabla\u_\eps}{r}^{\frac{r}{3r-6}+\frac{r(q-2)}{q(r-2)}}  \bigr)
\leq C\bigl( 1 + \ka^{1+\frac{4rq-6r}{3q(r-2)} \gamma (1-\frac{2}{r})}  \bigr)\leq C\bigl( 1 + \ka^{1+ \gamma (\frac{4}{3}-\frac{2}{q})}  \bigr).
\end{split}
\end{equation}
Finally, for the last part we get for $q\leq2$
\begin{equation*}
\lnorm{\omega_\eps^2}{q}\leq C\norm{\eps\Delta\ro_\eps \u_\eps}_{W^{-1,q}(\Omega)}\leq C\eps\bigl( \lnorm{\nabla\ro_\eps\u_\eps}{q} + \lnorm{\nabla\ro_\eps\nabla\u_\eps}{\frac{6}{5}} \bigr)
\leq C(\ka)\eps^{1/2}.\qedhere
\end{equation*}
\end{proof}

Now, we are approaching the key definition of the effective viscous flux. Inserting the Helmholtz decomposition into the approximative momentum equation yields
\begin{multline}
\nabla\bigl( -(2\nu+\eta)\Delta\mfi_\eps  + p_\eps(\ro_\eps,c_\eps)\bigr)=\nu \Delta \curl \A + K(\ro_\eps)\ro_\eps\ef - K(\ro_\eps)\ro_\eps\u_\eps\cdot \nabla\u_\eps - \frac{1}{2}\eps\Delta\ro_\eps \u_\eps  \\ 
-\frac{1}{2}\eps h K(\ro_\eps)\u_\eps  + \eps\frac{1}{2}\ro_\eps\u_\eps
+\sol\bigl(\nabla c_\eps\otimes\nabla c_\eps-\frac{\abs{\nabla c_\eps}^2}{2}\Id\bigr) - a\nabla c_\eps \int_{0}^{\ro_\eps}t K\carka(t)\de{t},
\end{multline}
and we introduce the fundamental quantity
\begin{equation}
G_\eps= -(2\nu+\eta)\Delta\mfi  + p_\eps(\ro_\eps,c_\eps) = -(2\nu+\eta)\sol\u_\eps  + p_\eps(\ro_\eps,c_\eps). \label{EFVe}
\end{equation}
Similarly, inserting the Helmholtz decomposition into the limit momentum equation  we obtain (with usage of the fact that due to the continuity equation we have
 $\overline{K(\ro)\ro\nabla\u}\u = \overline{K(\ro)\ro}\u\cdot \nabla\u $) that
\begin{multline}
\nabla\bigl( -(2\nu+\eta)\Delta\mfi  + \overline{p(\ro,c)}\:\bigr)=\nu \Delta \curl \A + \overline{K(\ro)\ro}\ef - \overline{K(\ro)\ro}\u\cdot \nabla\u \\ 
+\sol\bigl(\nabla c\otimes\nabla c-\frac{\abs{\nabla c}^2}{2}\Id\bigr) - a\nabla c \overline{\int_{0}^{\ro}t K\carka(t)\de{t}},
\end{multline}
hence we define the limit version of effective viscous flux by
\begin{equation}
G= -(2\nu+\eta)\Delta\mfi  + \overline{p(\ro,c)} = -(2\nu+\eta)\sol\u  + \overline{p(\ro,c)}.
\end{equation}
Note we have control of $\inte{G_\eps} = \inte{p_\eps(\ro_\eps,c_\eps)\:}$ and $\inte{G} = \inte{\overline{p(\ro,c)}\:}.$ Further we state the most important features of the effective viscous flux.
\begin{lem}\label{strongG}
There exists a subsequence  such that
$$ G_\eps\to G \text{ (strongly) in }L^{2}(\Omega) ,$$
and \begin{equation} \lnorm{G}{\infty}\leq C(\zeta) \bigl( 1+\ka^{1+\frac{2\gamma}{3} + \zeta} \bigr),\text{ for any }\zeta\in\Bigl(0,\frac{\gamma-6}{3}\Bigr].\label{zetao}\end{equation}
\end{lem}
\begin{proof}
Let us decompose $G_\eps$ to $G_\eps=G_\eps^1+G_\eps^2$, where $G_\eps^2$ contains the "strongly $\eps$-dependent" terms of the right hand side of \refx{EFVe}, namely
$\nabla G_\eps^2 =-\eps\frac{1}{2}\Delta \ro_\eps \u_\eps - \nu \curl \omega_\eps^{2} $ with $\inte{G_\eps^2}=0,$ so
$$ \lnorm{G_\eps^2}{q}\leq C(\eps \norm{\Delta\ro_\eps\u_\eps}_{W^{-1,q}(\Omega)} +\nu \norm{ \curl \omega_\eps^{2} }_{W^{-1,q}(\Omega)}) \leq C(\ka)\eps^{1/2},\text{ for }q\in[1,2].$$
Using once more the estimates from Lemma \ref{vort}, we observe that $\abs{\int_\Omega{G_\eps}\dx}\leq C $  and \begin{equation}
\norm{ G_\eps^1  }_{W^{1,q}(\Omega)}\leq C\bigl( 1 + \ka^{1+ \gamma (\frac{4}{3}-\frac{2}{q})}  \bigr),\:\text{ for } q\in \left[2,\frac{6\gamma}{6+\gamma}\right] .\label{zeta}
\end{equation}
Therefore, since $\gamma>6$ we have, at least for a suitably chosen subsequence
$$ G_\eps^1 \to G^{1} \text{ (strongly) in }L^{\infty}(\Omega)\:\text{ and }\: G_\eps^2 \to 0 \text{ (strongly) in }L^{2}(\Omega).$$
Thus, $$G_\eps =G_\eps^1+G_\eps^2 \to G^1 = G \text{ (strongly)  in }L^{q}(\Omega),\text{ for }q\in[1,2].$$
Finally, setting $q = 3+  \frac{3\zeta}{2\gamma-3\zeta}$ in \refx{zeta} we get the desired conclusion
$$  \lnorm{G}{\infty}\leq C(q)\norm{ G  }_{W^{1,q}(\Omega)}\leq C(\zeta)\bigl(1+\ka^{1+\frac{2\gamma}{3} + \zeta} \bigr) .\qedhere $$
\end{proof}
Now, we are ready to show that we are able to choose $\ka$ in such a way that actually $\overline{K(\ro)\ro} = \ro$ a.e. in $\Omega.$ This will be an immediate consequence of the following lemma.

\begin{lem}\label{kao}
There exists $\ka_0$ such that 
\begin{equation}
\frac{\ka-3}{\ka}(\ka-3)^\gamma - \lnorm{G}{\infty}\geq 1  \text{ for }\ka>\ka_0,\label{ka}
\end{equation}
and at least for a subsequence
$ \lim\limits_{\eps\to 0}\abs{\{ \ro_\eps>\ka-3\}}=0.$
\end{lem}
\begin{proof}
Let us define a smooth non-increasing function $N: [0,\infty)\to [0,1]$ such that
$$ N(t)=\begin{cases}
1&\text{ for }t\in[0,\ka-3],\\
\in[0,1]&\text{ for }t\in(\ka-3,\ka-2),\\
0&\text{ for }t\in[\ka-2,\infty),\\
\end{cases}$$
and multiply the approximative continuity equation by $N^l(\ro_\eps)$, for some suitable power $l\in\N$ in order to get after some manipulations
\begin{equation}
\inte{\int\limits_0^{\ro_\eps}  t l N^{l-1}(t)N\carka(t) \de{t}\sol \u_\eps }\geq R_\eps
\end{equation}
with $R_\eps = \eps \inte { N^{l}(\ro_\eps)\Delta\ro_\eps  + (h-\ro_\eps) N^{l}(\ro_\eps)}$, $R_\eps\to0$ as $\eps\to 0 $.\footnote{ 
Note that $-\eps l \inte{ N^{l-1}(\ro_\eps)N\carka(\ro_\eps)\abs{\nabla\ro_\eps}^2}\geq 0 $}
Further by definition of $G_\eps$ we have
\begin{multline}
-(\ka-3)\inte{ \biggl(\int\limits_0^{\ro_\eps} l N^{l-1}(t) N\carka(t) \de{t} \biggr) p_\eps(\ro_\eps,c_\eps)}\leq 
\ka \abs{\inte{  \biggl(\int\limits_0^{\ro_\eps} -l N^{l-1}(t) N\carka(t) \de{t} \biggr) G_\eps }} + R_\eps,
\end{multline}
and
\begin{equation}
\frac{\ka-3}{\ka} \int\limits_{\{\ro_\eps>\ka-3\}} \bigl(1-N^{l}(\ro_\eps)\bigr)p_\eps(\ro_\eps,c_\eps)\dx\leq \int\limits_{\{\ro_\eps>\ka-3\}} \bigl(1-N^{l}(\ro_\eps)\bigr) \abs{G_\eps}\dx + \abs{R_\eps}.
\end{equation}
Recalling the structure of the pressure we have according to \refx{logar} and the fact that $c_\eps\in[0,1]$
 \begin{equation}p_\eps(\ro_\eps,c_\eps)= P_b(\ro_\eps) +(ac_\eps+d)\int_0^{\ro_\eps} K(t) \det{\bigl(t^{2}l\carka_\eps(t)\bigr)}\:\de{t}\geq P_b(\ro_\eps), \label{press}
 \end{equation} 
 which yields \begin{multline}
\frac{\ka-3}{\ka}(\ka-3)^\gamma\abs{\{\ro_\eps>\ka-3\}}- \frac{\ka-3}{\ka} \lnorm{p_\eps(\ro_\eps,c_\eps)}{2}\lnorm{N^l(\ro_\eps)}{2}\\
\leq \lnorm{G}{\infty} \abs{\{\ro_\eps>\ka-3\}} +\lnorm{G-G_\eps}{1} + \abs{R_\eps}.
\end{multline}
According to \refx{zetao} we are able to choose $\ka_0$ satisfying \refx{ka}, yielding 
\begin{equation}
\abs{\{\ro_\eps>\ka-3\}}\leq C \bigl( \norm{N^l(\ro_\eps)}_{L^{2}(\{\ro_\eps>\ka-3\})} +\lnorm{G-G_\eps}{1} + \abs{R_\eps} \bigr).
\end{equation}
However, the last two terms tend to zero as $\eps\to 0+$ and having fixed $\eps>0$ $\norm{N^l(\ro_\eps)}_{L^{2}(\{\ro_\eps>\ka-3\})}$ tends to zero as $l\to+\infty$ as well. Thus, Lemma \ref{kao} is proved.
\end{proof}

Finally, we  deduce the pointwise convergence of the densities. Our main aim is to show that 
$$\overline{P_b(\ro)\ro} =\overline{P_b(\ro)} \ro, $$ which will further lead to $\ro_\eps\to\ro$ strongly in $L^{q}(\Omega)$ for $q<\infty.$
\begin{lem}
The weak limits satisfy
\begin{equation}\inte{\overline{p(\ro,c)\ro}}\leq \inte{G\ro}.\label{one}\end{equation}
\end{lem}
\begin{proof}
Testing the approximative continuity equation by $\log\ka -\log (\ro_\eps + \delta)$ and taking the limit for $\delta\to0+,$ we get (see \cite{MuPo06})
$$\inte{K(\ro_\eps)\u_\eps\cdot\nabla\ro_\eps}\geq \eps C(\ka)$$
and by Lemma \ref{kao} $ \inte{\ro_\eps\sol\u_\eps}\leq R_\eps,$
with $R_\eps\to 0 $ as $\eps\to 0+.$ Hence the definition of $G$ yields 
$$\inte{p_\eps(\ro_\eps,c_\eps)\ro_\eps} \leq \inte{G_\eps\ro_\eps} - (2\nu + \eta)R_\eps$$
 and passing to the limit with $\eps$ we get \refx{one}, since according to Lemma \ref{strongG} $ \overline{G\ro}=G\ro.$
\end{proof}

\begin{lem}
\begin{equation}
\inte{\overline{p(\ro,c)}\ro}= \inte{G\ro}.\label{two}
\end{equation}
\end{lem}
\begin{proof}
First, with the usage of Fridrichs' commutator lemma we are able to approximate $\ro$ by smooth bounded functions $\ro_n$ and deduce that $$\inte{\ro\sol\u + \u\cdot\nabla \ro} = 0.$$ Further 
testing continuity equation by $\log(\ro_n+\delta)-\log\delta$ and then passing to the limit at first with $n\to\infty$ and then with $\delta\to 0+$ gives us $\inte{\u\cdot\nabla\ro}=0,$ hence
 $\inte{\ro\sol\u}=0$ as well. Thus, using again the properties of $G$ we can conclude \refx{two}.

\end{proof}
Further, the strong convergence of $c$ and the convexity of $s\to s^{\gamma}$ and $s \to s^{2}l\carka_\eps(s) $ gives us $$ \overline{p(\ro,c)}\ro \leq \overline{p(\ro,c)\ro},$$ which combined 
with \refx{one} and \refx{two} yields
\begin{equation*}
\overline{p(\ro,c)\ro} =\overline{p(\ro,c)} \ro\: \text{ a.e. in }\Omega.
\end{equation*}
Thus, $\overline{\ro^{\gamma+1}} = \ro\overline{\ro^{\gamma}}$ and $\ro_\eps\to\ro$ strongly in $L^{\gamma}(\Omega).$

Now, we move our attention to the last two equations of system \refx{RK}-\refx{system} and show that due to strong convergence $\ro_\eps$ and $c_\eps$ all the remaining nonlinearities 
can be identified, so we have indeed obtained the solutions to our original system. This completes the proof of our main theorem.

\vspace*{0.35cm}

\begin{note}\label{smaller} In order to prove the existence of only weak solutions of the system under consideration for $6\geq\gamma>3$ we can use the following idea. We modify the
 pressure  $p_\delta(\ro,c) = p(\ro,c) + \delta \ro^\Gamma$ with $\Gamma>6$, so we can apply the already proven result and then using the a priori estimates derived in the heuristic 
approach pass to the limit with $\delta\to0+.$ This limit passage can be performed in the same spirit as in the case of the Navier-Stokes system, using nowadays "standard" techniques 
due to Lions, and Feireisl, see e.g.\cite{NoPo11} with $\theta$ replaced by $c$. The compactness of the additional stress in the momentum equation and in the additional equations is 
just easy application of the Rellich-Kondrachov compactness theorem due to the uniform bound of $\Delta c$.
\end{note}
\vspace*{0.35cm}

\textbf{Acknowledgments:} The work on this paper was conducted during the first author's internship at the Warsaw Center of Mathematics and Computer Science. 
The work of the first author was supported by the grant SVV-2014-267316. The second author was partly supported by the MN grant IdP2011 000661.

\providecommand{\bysame}{\leavevmode\hbox to3em{\hrulefill}\thinspace}
\providecommand{\MR}{\relax\ifhmode\unskip\space\fi MR }
\providecommand{\MRhref}[2]{%
  \href{http://www.ams.org/mathscinet-getitem?mr=#1}{#2}
}
\providecommand{\href}[2]{#2}

\end{document}